\documentclass[12pt,a4paper]{article}
\usepackage{mathrsfs}
\usepackage{amsfonts}
\usepackage{cases}
\usepackage{times}\usepackage{latexsym}
\usepackage{graphics}\usepackage{epsfig}
\usepackage{color}
\usepackage{amsthm}\usepackage{setspace}\usepackage{graphicx}\usepackage{float}\usepackage{subfigure}\usepackage{dsfont}\usepackage{fancyhdr}\usepackage{amssymb}
\newenvironment{keywords}{{\bf Keywords: }}{}

\newtheorem{remark}{Remark}[section]

\newtheorem{theorem}{Theorem}[section]

\newtheorem{definition}{Defintion}[section]

\begin{document}

\author{Hanbing Liu $^1$\\ { \small $^1$School of Mathematics and Physics of China University of
Geoscience,} \\{\small Wuhan, 430074, P.R.China. E-mail:
hanbing272003@aliyun.com}}
\title
 {Boundary sampled-data feedback stabilization for parabolic equations}
 \date{}
 \maketitle


 \hrule
\begin{abstract}
The aim of this work is to design an explicit finite dimensional boundary feedback controller of sampled-data form for locally exponentially stabilizing the equilibrium solutions to semilinear parabolic equations. The feedback controller is expressed in terms of the eigenfunctions corresponding to unstable eigenvalues of the linearized equation. This stabilizing procedure is applicable for any sampling rate, not necessary to be small enough, and it tends to the continuous-times version when the sampling period tends to zero.
\end{abstract}

 \noindent
 \keywords{Parabolic equations; sampled-data control; boundary feedback stabilization}

\section{Introduction}
\label{}
In this work, we aim to design an explicit finite-dimensional boundary feedback control of sampled-data form, to stabilize the equilibrium solutions $y_e\in C^2(\bar{\Omega})$ to the parabolic equation
\begin{equation}\label{e101}
 \left\{\begin{array}{ll}
 \frac{\partial y}{\partial t}= \Delta y+ f(x,y), \ \ \ \ \mathrm{in} \ (0, +\infty)\times\Omega,\\
 y=u \ \mathrm{on}\ (0, +\infty)\times\Gamma_1, y=0 \ \mathrm{ on}\ I\times\Gamma_2,\\
  y(0,x)=y_0(x),\ \ \ \mathrm{in} \
  \Omega.
\end{array}\right.
\end{equation}
where $\Omega$ is a bounded and open domain of $\mathbb{R}^d$ with smooth boundary $\partial \Omega=\Gamma_1\cup\Gamma_2$, $\Gamma_1, \Gamma_2$ being connected parts of $\partial \Omega$.
$T>0$ is the sampling period. By sampled-data control, we mean that the control is time-discrete function. More exactly, it is of the form
\begin{equation}\label{e102}
u(t,x)=\sum_{i=1}^{\infty}\chi_{[iT,(i+1)T)}(t)u_i, u_i\in L^2(\Gamma_1).
\end{equation}
Here $\chi_{[iT,(i+1)T)}(\cdot)$ is the characteristic function of interval $[iT,(i+1)T)$ for $i=1,2, \cdots$. The positive number $T$ is called the sampling period.

The existence of a boundary controller which stabilize the linear parabolic equation was firstly established in \cite{01}, where a conceptual procedure to construct the boundary controller was given, but without giving an explicit form. V. Barbu firstly introduced in his works \cite{02, 03, 04} a technique to construct a stabilizing boundary feedback controller, which is of finite dimensional and in an explicit form. However, these works based on an assumption that the normal derivatives of the eigenfunctions corresponding to the unstable eigenvalues are linearly independent. This method was used to construct explicit controllers to stabilize different kinds of equations in \cite{06, 07, 08}. Later on, I. Munteanu developed a delicate approach in \cite{09} to construct a similar type of boundary feedback controller to stabilize the parabolic equations, where he dropped the assumption imposed by V. Barbu. H. Liu et. al. applied this control to Fisher's equation in \cite{091}, and proved that it can locally stabilize this kind of semilinear parabolic equations. In last decade, a different approach to construct explicit boundary feedback controller to stabilize the 1-D linear and nonlinear parabolic equations, which is the so-called backstepping method, was developed. We cite here \cite{10, 11, 12,13}.

The adapting of the sampled-data control becomes more and more popular with the developing of digital technique. Sampled-data feedback stabilization for linear and nonlinear parabolic equations were studied in \cite{14, 15, 16, 17}. Among these works, \cite{14} showed that two kinds of sampled-data boundary feedback, which emulate the reduced model design and the back-stepping design respectively, can stabilize the 1-D linear parabolic equations when the length of the sampling interval is small enough; \cite{15, 16} concerned about the length of sampling interval that preserves the stability of the closed-loop system with proportional feedback; while \cite{17} provided a way to construct an output feedback control to stabilize the heat equations for arbitrarily given sampling period.

In this work, we shall develop the technique introduced in \cite{09} and \cite{091} to construct a sampled-data feedback control for stabilizing the semilinear parabolic equations. Several novelties of this work should be stressed. Firstly, the equation under consideration is of multi-dimension, and with polynomial-like nonlinearity, while the works \cite{14, 15, 16} considered 1-D linear equations or semilinear equations where the nonlinear part is globally Lipschiz; Secondly, the feedback is of finite dimensional and in an explicit form, and it stabilizes the equation for arbitrarily given sampling period, not necessarily to be small enough as that in  \cite{14, 15, 16}; Lastly, we can see the behavior of the feedback with respect to the sampling period, and the sampled-data feedback control tends to the continuous-time feedback control constructed in \cite{09} and \cite{091} when the sampling period goes to zero.

The rest of this paper is organized as follows. In Section 2, we shall construct the feedback control, and prove that it stabilizes the linearized equation. In Section 3, we present our main result, that is, the feedback control constructed in Section 2 also stabilizes locally the semilinear parabolic equations. In Section 4, the numerical examples will be given and the dependence of the sampling period will be analyzed.

\section{Stabilization of the linearized equation}
\setcounter{equation}{0}
\subsection{Notations and the well-posedness of the equation}
Everywhere in the following, we shall assume that\\
(i) $f, f_y \in C(\bar{\Omega}, \mathbb{R}).$\\
In particular, this implies that $x\rightarrow f_y(x,y_e(x))$ is continuous in $\bar{\Omega}$.

We define the linear operator $A: D(A)(:= H^2(\Omega)\cap H_0^1(\Omega))\rightarrow L^2(\Omega)$ by
$$Ay=-\Delta y-f_y(x,y_e)y, \forall y\in D(A).$$
We assume that the operator $A$ has at least one negative eigenvalue. Since the resolvent of $A$ is compact, it has countable set of eigenvalues. Given $\rho>0$, let $\{\lambda_j\}_{j=1}^\infty$, with
\begin{equation}\label{e202}
\lambda_1\leq\lambda_2\leq\cdots \leq\lambda_N< \rho\leq \lambda_{N+1}\leq\cdots
\end{equation}
 be the family of all  eigenvalues of $A$ and let
 $\{\phi_j\}_{j=1}^\infty$ be the family of the corresponding
 eigenfunctions, which forms an  orthonormal basis of $L^2(\Omega)$.

In the rest of the paper, we shall denote by $\|\cdot\|$, $\|\cdot\|_s$, $|\cdot|_0$ and $|\cdot|_N$ the norms of $L^2(\Omega)$, $H^s(\Omega)$, $L^2(\Gamma_1)$ and  $\mathds{R}^N$ respectively. The inner products in $L^2(\Omega)$, $L^2(\Gamma_1)$ and Euclid space $\mathds{R}^N$ will be denoted by $\langle \cdot, \cdot \rangle$, $\langle\cdot, \cdot\rangle_0$  and  $\langle \cdot, \cdot \rangle_N$, respectively. We shall write $Q=\Omega\times (0,\infty)$ for simplicity, and the variables $x,t$ will be omitted in the case of no ambiguity. For each $M\in\mathds{N}^+$,  let $X_M=\mathrm{span}\{\phi_i\}_{i=1}^M$, and let $P_M$
   be the orthogonal projection from $L^2(\Omega)$ onto $X_M$. We denote by $Q_N: L^2(\Omega)\rightarrow \mathbb{R}^N$ the operator $Q_N(y)=(\langle y, \phi_1\rangle, \cdots, \langle y, \phi_N\rangle)'$. Here $B'$ stands for the transposition of the matrix of $B$.

We give firstly the notion of the solution to the linearized equation with sampled-data Dirichlet boundary condition.
\begin{definition}\label{def01}
Let $\tilde{y}_0\in L^2(\Omega)$, and $v(t)=\sum_{i=1}^{\infty}\chi_{[iT,(i+1)T)}(t)v_i$, $v_i\in L^2(\Gamma_1)$ be given. A solution of the equation
\begin{equation}\label{e201}
 \left\{\begin{array}{ll}
 \frac{\partial y}{\partial t}= \Delta y+ f_y(x,y_e)y, \ \ \ \ \mathrm{in} \ (0,\infty)\times\Omega,\\
 y=v \ \mathrm{on}\ (0,\infty)\times\Gamma_1, y=0 \ \mathrm{ on}\ (0,\infty)\times\Gamma_2,\\
  y(0,x)=\tilde{y}_0(x),\ \ \ \mathrm{in} \
  \Omega.
\end{array}\right.
\end{equation}
is a function $y\in C_b([0, +\infty); L^2(\Omega))$, such that, for every $\tau\in [0, +\infty)$ and for every $\zeta\in W^{1,2}([0,T]; L^2(\Omega))\cap (L^2(0, T); D(A))$, one has
\begin{eqnarray}\label{e2001}
&&-\int_0^\tau\int_\Omega(\zeta_t-A\zeta)ydxdt-\int_0^\tau\int_{\Gamma_1}v\frac{\partial\zeta}{\partial n}dxdt\nonumber\\
&&+\int_\Omega y(\tau)\zeta(\tau)dx-\int_\Omega \tilde{y}_0\zeta(0)dx=0.
\end{eqnarray}
\end{definition}

For the well-posedness of (\ref{e201}), we have the following result.
\begin{theorem}\label{th201}
Under assumption (i), for given $\tilde{y}_0\in L^2(\Omega) $, and discrete-time function $v(t,x)=\sum_{i=0}^{\infty}\chi_{[iT,(i+1)T)}(t)v_i, v_i\in L^2(\Gamma_1)$, the equation (\ref{e201}) has a unique solution.
\end{theorem}
\begin{proof}
The solution can be equivalently defined as: $y\in C_b([0, +\infty); L^2(\Omega))$ satisfies, for every $i\in \mathbb{N}$ and every $\tau\in [iT, (i+1)T]$, every $\zeta\in W^{1,2}([0,\tau]; L^2(\Omega))\cap L^2((0, \tau); D(A))$,
\begin{eqnarray}\label{e2002}
&&-\int_{iT}^\tau\int_\Omega(\zeta_t-A\zeta)ydxdt-\int_{iT}^\tau\int_{\Gamma_1}v\frac{\partial\zeta}{\partial n}dxdt\nonumber\\
&&+\int_\Omega y(\tau)\zeta(\tau)dx-\int_\Omega y(iT)\zeta(iT)dx=0.
\end{eqnarray}
Consider the equation (\ref{e2002}) firstly on the interval $[0,T]$, i.e., $i=0$. We prove firstly the uniqueness of the solution. Suppose that $y_1$ and $y_2$ are two solutions. Let $\tilde{y}=y_1-y_2$. Then, for every $\tau\in [0, T]$, one has
\begin{equation}\label{e2003}
-\int_{0}^\tau\int_\Omega(\zeta_t-A\zeta)\tilde{y}dxdt+\int_\Omega \tilde{y}(\tau)\zeta(\tau)dx=0.
\end{equation}
Let $\tilde{h}_n\in D(A)$ and $\tilde{h}_n\rightarrow \tilde{y}(\tau)$ in $L^2(\Omega)$, and take $\zeta\in  W^{1,2}([0,\tau]; L^2(\Omega))\cap L^2((0, \tau); D(A))$ satisfying that
\begin{equation}\label{e2004}
\zeta_t-A\zeta=0, t\in(0, \tau);\ \  \zeta(\tau)=\tilde{h}_n.
\end{equation}
Then, it follows from (\ref{e2003}) that $\int_\Omega \tilde{y}(\tau)\tilde{h}_ndx=0$. Letting $n\rightarrow \infty$, we obtain that $\int_\Omega |\tilde{y}(\tau)|^2dx=0$, which implies that $\tilde{y}(\tau)=0, \forall \tau\in [0, T]$. Hence, the solution is unique.

Now, we prove the existence. Let $v^0(t)=v_0, \forall t\in[0,T]$. It is well-know that for sufficiently large $k>0$, the solution to the equation
 \begin{equation}\label{e2005}
k\psi- \Delta \psi+ f_y(x,y_e)\psi=0,\ \mathrm{in}\ \Omega;\ \  \psi=v^0\ \mathrm{on}\ \Gamma_1, \psi=0 \ \mathrm{ on}\ \Gamma_2,
\end{equation}
exists, and $\psi\in C^1([0,T]; L^2(\Omega))$. Let $z\in C([0,T]; L^2(\Omega))$ be the solution to the equation
\begin{equation}\label{e2006}
z_t+Az=k\psi, t\in(0, \tau);\ \  z(0)=y_0-\psi(0).
\end{equation}
Then, by a direct calculation, one can show that $y=z+\psi\in C([0, T]; L^2(\Omega))$ satisfies the equation (\ref{e2002}) with $i=0$. This implies the existence of the solution. Moreover,
we can see that $\|y(\tau)\|\leq C(\|v_0\|+\|y_0\|), \forall \tau\in [0,T]$. Step by step, we can show that there exists unique solution $y\in C_b([0, +\infty); L^2(\Omega))$,  for every $i\in \mathbb{N}$, satisfying (\ref{e2002}).
\end{proof}

\subsection{The stabilization of the linearized equation}
We introduce firstly the feedback. Let $\rho<\gamma_1<\gamma_2<\cdots<\gamma_N$ be $N$ constants, where $\rho$ is given by (\ref{e202}). It is not difficult to show that, $\exists\sigma>0$, such that $\forall w\in D(A)$,
$$\langle\sum_{i=1}^N(\frac{e^{-\lambda_iT}-e^{-\gamma_kT}}{\int_0^Te^{-\lambda_is}ds}-\lambda_i)\langle w, \phi_i\rangle\phi_i+ Aw, w\rangle\geq \sigma \|w\|_1^2.$$
Then, for each $k\in\{1, 2, \cdots, N\}$, the solution to the equation
\begin{equation}\label{e203}
\left\{\begin{array}{ll}
\sum_{i=1}^N(\frac{e^{-\lambda_iT}-e^{-\gamma_kT}}{\int_0^Te^{-\lambda_is}ds}-\lambda_i)\langle\psi_k, \phi_i\rangle\phi_i-\Delta\psi_k-f_y(x,y_e)\psi_k=0,  \\ \ \mathrm{in} \ \Omega,
\psi_k=v, \mathrm{on}\ \Gamma_1, \psi_k=0 \ \mathrm{ on}\ \Gamma_2,
\end{array}\right.
\end{equation}
exists for any given $v\in L^2(\Gamma_1)$. We shall denote by $D_{\gamma_k}$ the map $:v\rightarrow \psi_k(\cdot)$, i.e., $\psi_k(\cdot)=D_{\gamma_k} v$. It is known that $\psi_k\in H^{1/2}(\Omega)$ and $\|\psi_k\|_{1/2}\leq C|v|_{0}$ (see \cite{18}).

We introduce the matrices
\begin{equation}\label{e204}
\Lambda_{\gamma_k}:=\mathrm{diag}(\frac{\int_0^Te^{-\lambda_is}ds}{e^{-\lambda_iT}-e^{-\gamma_kT}})_{1\leq i\leq N}, \Lambda=\sum_{k=1}^N\Lambda_{\gamma_k},
\end{equation}
 and
\begin{equation}\label{e205}
B_0=\left(\langle\frac{\partial\phi_i}{\partial n}, \frac{\partial\phi_j}{\partial n}\rangle_0 \right)_{1\leq i, j\leq N}, B=(B_1+B_2+\cdots+B_N)^{-1},
\end{equation}
where $B_k=\Lambda_{\gamma_k}B_0\Lambda_{\gamma_k}, k=1,2, \cdots, N$. Following the method in \cite{09} one can show that $B_1+B_2+\cdots+B_N$ is invertible.

Now, we give the feedback $F: L^2(\Omega)\rightarrow L^2(\Gamma_1)$ by
\begin{eqnarray}\label{e206}
&&F(w)=\mathds{1}_{\Gamma_1}\left\langle\Lambda BQ_N(w), \frac{\partial\mathbf{\Phi}^N}{\partial n}\right\rangle_N.
\end{eqnarray}
where $\frac{\partial\mathbf{\Phi}^N}{\partial n}=(\frac{\partial\phi_1}{\partial n}, \frac{\partial\phi_2}{\partial n}, \cdots, \frac{\partial\phi_N}{\partial n})'\in (L^2(\partial\Omega))^N$, and $\mathds{1}_{\Gamma_1}: L^2(\partial\Omega)\rightarrow L^2(\Gamma_1)$ is the restrictive operator. We shall denote by $F_k: L^2(\Omega)\rightarrow L^2(\Gamma_1), k=1, 2, \cdots, N$ the operators
\begin{eqnarray}\label{e206}
&&F_k(w)=\mathds{1}_{\Gamma_1}\left\langle \Lambda_{\gamma_k}BQ_Nw, \frac{\partial\mathbf{\Phi}^N}{\partial n}\right\rangle_N.
\end{eqnarray}
Then $F=\sum_{k=1}^NF_k$. The sampled data feedback control we design is given by
\begin{equation}\label{e207}
v(t)=\sum_{i=0}^{\infty}\chi_{[iT,(i+1)T)}(t)Fy(iT).
\end{equation}
\begin{remark}
The feedback $F$ we constructed here depends on the sampling period $T$, since the matrices $\Lambda_{\gamma_k}, k=1, 2, \cdots, N$, depend on $T$. We can see that, when the sampling period $T\rightarrow 0$, the elements in $\Lambda_{\gamma_k}$ satisfy  $\frac{\int_0^Te^{-\lambda_is}ds}{e^{-\lambda_iT}-e^{-\gamma_kT}}\rightarrow \frac{1}{\gamma_k-\lambda_i}$. Therefore, $\Lambda_{\gamma_k}\rightarrow \Lambda^0_{\gamma_k}$ when $T\rightarrow 0$, where $\Lambda^0_{\gamma_k}=\mathrm{diag}\{\frac{1}{\gamma_k-\lambda_i}\}_{1\leq i\leq N}$. Write $\Lambda^0=\sum_{k=1}^N\Lambda^0_{\gamma_k}, B^0=(\sum_{k=1}^N\Lambda^0_{\gamma_k}B_0\Lambda^0_{\gamma_k})^{-1}$. The feedback $F$ defined in (\ref{e206}) will tend to
$F^0: L^{\Omega}\rightarrow L^2(\Gamma_1)$, which is defined as follows
$$F^0(w)=\mathds{1}_{\Gamma_1}\left\langle\Lambda^0 B^0Q_N(w), \frac{\partial\mathbf{\Phi}^N}{\partial n}\right\rangle_N.$$
The above feedback is exactly the same as that given in \cite{09} and \cite{091}, which was used to construct a continuous-time boundary feedback control. 
\end{remark}
The following result amounts to saying that the feedback $v(t)$ achieves global exponential stability of the linearized system. More precisely,
\begin{theorem}\label{th202}
Assume that $\tilde{y}_0\in L^2(0,1)$. The feedback control $v(t)$, given by (\ref{e207}), exponentially stabilizes the linearized equation. More exactly, there exist constants $C>0$, such that the solution to equation
\begin{equation}\label{e208}
 \left\{\begin{array}{ll}
 \frac{\partial y}{\partial t}= \Delta y+ f_y(x,y_e)y, \ \ \ \ \mathrm{in} \ \mathbb{R}^+\times\Omega,\\
 y=\sum_{i=0}^{\infty}\chi_{[iT,(i+1)T)}(t)Fy(iT)\ \mathrm{on}\ \mathbb{R}^+\times\Gamma_1, \\
  y=0 \ \mathrm{ on}\ I\times\Gamma_2, y(0,x)=\tilde{y}_0(x),\ \ \ \mathrm{in} \
  \Omega.
\end{array}\right.
\end{equation}
 satisfies
\begin{equation}\label{e209}
\|y(t)\|\leq Ce^{-\rho t}\|y(0)\|, \forall t\geq 0.
\end{equation}
\end{theorem}
\begin{proof}
We firstly lift the boundary. Write
\begin{eqnarray}\label{e210}
v_k(t)=\sum_{i=0}^{\infty}\chi_{[iT,(i+1)T)}(t)F_k(y(iT)),
\end{eqnarray}
for each $k=1, 2, \cdots, N $. Denote by $h_k$ the solution to equation (\ref{e203}) with boundary value $v=v_k$, i.e.
\begin{equation}\label{e211}
h_k=D_{\gamma_k}v_k, k=1, 2, \cdots, N.
\end{equation}
Set $z(t,x)=y(t,x)-\sum_{j=1}^Nh_j(t,x)$. Denote by $\mathbf{y}^N$, $\mathbf{z}^N$ and $\mathbf{h_j}^N$ the respective vectors $Q_Ny$, $Q_Nz$ and $Q_Nh_j$.
Involving equation (\ref{e203}) and the definition of $\phi_i$, by simple calculation, we can get that, for $1\leq i, k\leq N$,
\begin{equation}\label{e214}
\langle h_k, \phi_i\rangle=-\frac{\int_0^Te^{-\lambda_is}ds}{e^{-\lambda_iT}-e^{-\gamma_kT}}\langle v_k(t), \frac{\partial\phi_i}{\partial n}\rangle_0.
\end{equation}
With this identity, and the definition of $v_k(t)$, it follows that
\begin{equation}\label{e215}
\mathbf{h}_k^N(t)=-\sum_{i=0}^{\infty}\chi_{[iT,(i+1)T)}(t)B_kB\mathbf{y}^N(iT).
\end{equation}
By the latter equation, and the relation between $y$ and $z$, one can obtain that
\begin{equation}\label{e216}
\mathbf{y}^N(iT)=\frac{1}{2}\mathbf{z}^N(iT), \forall i=0, 1,  \cdots,
\end{equation}
and, for $k=1, 2, \cdots, N$,
 \begin{equation}\label{e218}
\mathbf{h}_k^N(t)=-\frac{1}{2}\sum_{i=0}^{\infty}\chi_{[iT,(i+1)T)}(t)B_kB\mathbf{z}^N(iT).
\end{equation}
Moreover, by the definition of $v_k$, we have
\begin{eqnarray}\label{e217}
v_k(t)=\sum_{i=0}^{\infty}\chi_{[iT,(i+1)T)}(t)\tilde{F}_k(z(iT)).
\end{eqnarray}
where $\tilde{F}_k: L^2(\Omega) \rightarrow L^2(\Gamma_1)$ is the operator given by
 \begin{equation}\label{e2117}
\tilde{F}_k(w)=\frac{1}{2}\mathds{1}_{\Gamma_1}\sum_{i=0}^{\infty}\chi_{[iT,(i+1)T)}(t)\langle \Lambda_{\gamma_k}BQ_Nw, \frac{\partial \Phi^N}{\partial n}\rangle_N.
\end{equation}

It is not difficult to see that $z$ is the solution to the following impulse evolution equation with homogeneous Dirichlet boundary condition
\begin{equation}\label{e212}
 \left\{\begin{array}{ll}
 \frac{\partial z}{\partial t}+\mathcal{A}z(t)=R(h_1, h_2, \cdots, h_N), \mathrm{in}\ \mathbb{R}^+\setminus\cup_{i=0}^\infty\{iT\},\\
 z(iT)=z(iT-)+\sum_{j=1}^Nh_j((i-1)T)-\sum_{j=1}^Nh_j(iT),\\
\ \ i=0, 2,\cdots,\ \ z(0)=y(0)-\sum_{j=1}^Nh_j(0).
\end{array}\right.
\end{equation}
Here
\begin{eqnarray}\label{e213}
R(h_1, h_2, \cdots, h_N)=\sum_{i, k=1}^N(\frac{e^{-\lambda_iT}-e^{-\gamma_kT}}{\int_0^Te^{-\lambda_is}ds}-\lambda_i)\langle h_k, \phi_i\rangle\phi_i
\end{eqnarray}
We shall write this term for simplicity by $R$ in the following. The second identity of (\ref{e212}) holds because $y$ is continuous, and by the definition of $h_k$, it can be understood in the equivalent manner:
\begin{eqnarray}\label{e2113}
z(iT)=(I-\sum_{k=1}^ND_{\gamma_k}F_k)(z(iT-)+\sum_{j=1}^Nh_j((i-1)T))).
 \end{eqnarray}

Denote by $A_N$ the diagonal matrix $\mathrm{diag}(\lambda_i)_{1\leq i\leq N}$. Multiplying equation (\ref{e211}) by $\phi_1, \phi_2, \cdots, \phi_N$ respectively, and using identity (\ref{e215}) and (\ref{e218}), one can obtain the equation satisfied by $\mathbf{z}^N$ as follows,
\begin{equation}\label{e219}
 \left\{\begin{array}{ll}
 \frac{d}{dt}\mathbf{z}^N(t)+A_N\mathbf{z}^N(t) \\
= -\frac{1}{2}\sum_{i=0}^{\infty}\chi_{[iT,(i+1)T)}(t)\sum_{k=1}^N(\Lambda_{\gamma_k}^{-1}-A_N)B_kB\mathbf{z}^N(iT),\\
 \ \ \ \ \ \ \ \ t\in(0, +\infty)\setminus\cup_{i=1}^\infty\{iT\},\\
 \mathbf{z}^N(iT)=2\mathbf{z}^N(iT-)-\mathbf{z}^N((i-1)T),\ \ i=1, 2,\cdots\\ \mathbf{z}^N(0)=2\mathbf{y}^N(0).
\end{array}\right.
\end{equation}
Then, for each $i\in \mathbb{N}$, by a first equation in (\ref{e219}) on the interval $[iT, (i+1)T)$, we can obtain by a direct calculation that
\begin{eqnarray}\label{e220}
&&\mathbf{z}^N(((i+1)T)-)\nonumber\\
&=&\frac{1}{2}\sum_{k=1}^N(e^{-\gamma_kT}+1)B_kB\mathbf{z}^N(iT).
\end{eqnarray}
It follows by the above equation and the second equation in (\ref{e219}) that
\begin{eqnarray}\label{e222}
\mathbf{z}^N((i+1)T)=\sum_{k=1}^Ne^{-\gamma_kT}B_kB\mathbf{z}^N(iT), i=0, 1, 2, \cdots.
\end{eqnarray}
Since for each $k=1, 2, \cdots, N$, the matrix $(e^{-\gamma_1T}-e^{-\gamma_kT})B^{1/2}B_kB^{1/2}$ is semipositive, we know that the sum of them is also semipositive. Notice that  $\lambda_{max}(\sum_{k=1}^Ne^{-\gamma_1T}B^{1/2}B_kB^{1/2})=e^{-\gamma_1T}$,  it follows that (see \cite{19} )
\begin{eqnarray}\label{e223}
\lambda_{max}(\sum_{k=1}^Ne^{-\gamma_kT}B^{1/2}B_kB^{1/2})\leq e^{-\gamma_1T}.
\end{eqnarray}
This together with (\ref{e222}) imply that, for each $i\in \mathbb{N}$
\begin{eqnarray}\label{e224}
|B^{1/2}\mathbf{z}^N((i+1)T)|_N\leq e^{-\gamma_1 T}|B^{1/2}\mathbf{z}^N(iT)|_N.
\end{eqnarray}
Hence, for $ i=1, 2, \cdots,$
 \begin{eqnarray}\label{e225}
&&|\mathbf{z}^N(iT)|_N
\leq C|B^{1/2}\mathbf{z}^N(iT)|_N\nonumber\\
&&\leq Ce^{-\gamma_1i T}|B^{1/2}\mathbf{z}^N(0)|_N
\leq Ce^{-\gamma_1 iT}|\mathbf{z}^N(0)|_N.
\end{eqnarray}
The constant $C$ in the right hand side represent the ratio of $\lambda_{max}(B^{1/2})$ to  $\lambda_{min}(B^{1/2})$. It depends on $N, \Omega$, $\omega$ and the set of values $\{\gamma_k\}_{k=1}^N$.

Now, for each $t>0$, there exists $i\in \mathbb{N}^+$, such that $t\in [iT, (i+1)T)$, and
 \begin{eqnarray}\label{e226}
|\mathbf{z}^N(t)|_N&=& e^{-A_N (t-iT)}\mathbf{z}^N(iT)\nonumber\\
&&-\frac{1}{2}\int_{iT}^{t}\sum_{k=1}^Ne^{-A_N (t-s)}ds(\Lambda_{\gamma_k}^{-1}-A_N)B_kB\mathbf{z}^N(iT)\nonumber\\
&\leq& C|\mathbf{z}^N(iT)|_N
\leq Ce^{-\gamma_1i T}|\mathbf{z}^N(0)|_N\nonumber\\
&\leq& Ce^{-\gamma_1t}|\mathbf{z}^N(0)|_N.
\end{eqnarray}
Hence, the first $N$ modes of $z$ is stable. It follows immediately by the definition of $h_j, j=1, 2, \cdots, N$, that
\begin{equation}\label{e227}
\|h_j(t)\|\leq Ce^{-\gamma_1t}|\mathbf{z}^N(0)|_N, \forall t>0.
 \end{equation}

Now we consider the high frequency part $\mathbf{z}^s(t)=(I-P_N)z$. Denote by $A_s=(I-P_N)A$. Given $t>0$, there exists $i\in \mathbb{N}^+$, such that $t\in [iT, (i+1)T)$, then, we can see from equation (\ref{e212}) that
 \begin{eqnarray}\label{e228}
&&\mathbf{z}^s(t)=e^{-A_s (t-iT)}\mathbf{z}^s(iT)+\int_{iT}^{t}e^{-A_s (t-s)}(I-P_N)Rds\nonumber\\
&&= e^{-A_s (t-(i-1)T)}\mathbf{z}^s((i-1)T)+\int_{(i-1)T}^{t}e^{-A_s (t-s)}(I-P_N)Rds\nonumber\\
&&+e^{-A_s (t-iT)}[\sum_{j=1}^N(I-P_N)(h_j((i-1)T)-h_j(iT))].
\end{eqnarray}
Do step by step as above, it follows that
 \begin{eqnarray}\label{e229}
&&\mathbf{z}^s(t)=e^{-A_s t}\mathbf{z}^s(0) +\int_{0}^{t}e^{-A_s (t-s)}(I-P_N)Rds\nonumber\\
&&+\sum_{k=1}^ie^{-A_s (t-kT)}[\sum_{j=1}^N(I-P_N)(h_j((k-1)T)-h_j(kT))].\nonumber\\
\ \
\end{eqnarray}
Using (\ref{e227}), and notice that $\rho<\min\{\gamma_1, \lambda_{N+1}\}$, we can obtain that
\begin{eqnarray}\label{e230}
&&\|\mathbf{z}^s(t)\|\nonumber\\
&\leq& e^{-\lambda_{N+1} t}\|\mathbf{z}^s(0)\|+Ce^{-\rho t}\sum_{k=1}^ie^{-(\gamma_1-\rho)kT}|\mathbf{z}^N(0)|_N\nonumber\\
&&+ C\int_{0}^{t}e^{-\rho (t-s)}e^{-\gamma_1s}ds\nonumber\\
&\leq& Ce^{-\rho t}\|z(0)\|, \forall t>0.
\end{eqnarray}
This,  together with (\ref{e226}) imply that
\begin{equation}\label{e231}
\|z(t)\|\leq Ce^{-\rho t}\|z(0)\|, \forall t>0.
 \end{equation}
Finally, by (\ref{e227}), (\ref{e231}) and the relation between $y$ and $z$, we obtain the estimate
\begin{eqnarray}\label{e224}
\|y(t)\|\leq C e^{-\rho t}\|y(0)\|, t\geq 0.
\end{eqnarray}
 This completes the proof of Theorem \ref{th201}.
\end{proof}

\section{Stabilization of nonlinear equation}
\setcounter{equation}{0}
Let $\varepsilon\in (0, 1/2)$ be an arbitrarily given number.    Assume that
$$(H_f) \ \ \ |f(x, y+y_e)-f(x, y_e)-f_y(x, y_e)y|\leq C\sum_{i=1}^l|y|^{p_i},$$
where $l$ is a positive integer, $p_1\leq p_2\leq \cdots\leq  p_l$, and $p_i, i=1, 2, \cdots, l$ satisfy that
$$0<p_i<1/\varepsilon, \mathrm{if}\ n=1;\ 0<p_i<\frac{n+1+2\varepsilon}{n-1+2\varepsilon}, \mathrm{if}\ n>1.$$

 For the stabilization of the nonlinear parabolic equation, we have the following result.
\begin{theorem}\label{th301}
Given $0<\mu<\rho$, under assumption $(H_f)$, when $y_0\in H^{1/2-\varepsilon}(\Omega)$ and $\|y_0-y_e\|_{1/2-\varepsilon}$ small enough, the feedback $u$, given by (\ref{e207}), locally stabilizes the equation (\ref{e101}). More exactly, there exist constants $C>0$,  and $\delta>0$, such that for all $y_0\in H^{1/2-\varepsilon}(\Omega)$ satisfying $\|y_0-y_e\|_{1/2-\varepsilon}\leq \delta$, the solution to equation
\begin{equation}\label{e301}
 \left\{\begin{array}{ll}
 \frac{\partial y}{\partial t}= \Delta y+ f(x,y),  \ \ \ \mathrm{in} \ (0,\infty)\times\Omega,\\
 y=\sum_{i=0}^{\infty}\chi_{[iT,(i+1)T)}(t)F(y(iT)-y_e)+y_e,\\
 \mathrm{on}\ (0,\infty)\times\Gamma_1,  y=0 \ \mathrm{ on}\ (0,\infty)\times\Gamma_2,\\
  y(0,x)=y_0(x),\ \ \ \mathrm{in} \
  \Omega.
\end{array}\right.
\end{equation}
 satisfies
\begin{equation}\label{e302}
\|y(t)-y_e\|_{1/2-\varepsilon}\leq Ce^{-\mu t}\|y_0-y_e\|_{1/2-\varepsilon}, t\geq 0.
\end{equation}
\end{theorem}
\begin{proof}
We divide the proof into the following four steps.

{\bf Step 1. Translate the equation into an equation with homogeneous Dirichlet boundary condition.} We do the substitution $y-y_e\rightarrow y$. Then, to prove inequality (\ref{e302}) for the solution to (\ref{e301}), it suffices to prove that there exist constants $C>0, \mu>0$ and $\delta>0$, such that for all $H^{1/2-\varepsilon}(\Omega)$ satisfying $\|y(0)\|_{1/2-\varepsilon}\leq \delta$, the the solution to equation
\begin{equation}\label{e303}
 \left\{\begin{array}{ll}
 \frac{\partial y}{\partial t}= \Delta y+ f_y(x,y_e)y+f(x,y+y_e)-f(x,y_e) \\  \ \ \  \ \ \ -f_y(x,y_e)y,\ \ \ \mathrm{in} \ (0,\infty)\times\Omega,\\
 y=\sum_{i=0}^{\infty}\chi_{[iT,(i+1)T)}(t)Fy(iT),
 \mathrm{on}\ (0,\infty)\times\Gamma_1,\\  y=0 \ \mathrm{ on}\ (0,\infty)\times\Gamma_2,
  y(0)=y_0-y_e,\ \mathrm{in} \
  \Omega,
\end{array}\right.
\end{equation}
satisfies
\begin{equation}\label{e304}
\|y(t)\|_{1/2-\varepsilon}\leq Ce^{-\mu t}\|y(0)\|_{1/2-\varepsilon}, \forall t\geq 0.
\end{equation}
Let $z(t,x)= y(t,x)-\sum_{j=1}^Nh_j(t,x)$, where $h_j=D_{\gamma_j}v_j$, and $v_j=\sum_{i=0}^{\infty}\chi_{[iT,(i+1)T)}(t)F_j(y(iT))$. Then, by (\ref{e217}), it follows that
\begin{eqnarray}\label{e305}
y(t,x)&=&z(t,x) +\sum_{j=1}^Nh_j(t,x)\nonumber\\
&=&z(t,x)+\sum_{i=0}^{\infty}\chi_{[iT,(i+1)T)}(t)\sum_{j=1}^N\tilde{F}_j(z(iT)).
\end{eqnarray}
Here $\tilde{F}_j, j=1, 2, \cdots $  are defined in (\ref{e2117}). Then, the equation of $y$ is equivalent to the following equation of $z$,
\begin{equation}\label{e306}
 \left\{\begin{array}{ll}
 \frac{\partial z}{\partial t}+Az= g(z+\sum_{i=1}^{\infty}\chi_{[iT,(i+1)T)}(t)\sum_{j=1}^N\tilde{F}_j(z(iT))), \\  \ \ \  \ \ \ +R(h_1, h_2, \cdots, h_N)  \ \ \ \mathrm{in} \ \mathbb{R}^+\setminus\{iT\}_{i=1}^\infty\times\Omega,\\
  z(0,x)=y_0(x)-\sum_{j=1}^NF_j(y(0)),\ \ \ \mathrm{in} \
  \Omega.\\
  z(iT)=z(iT-)+\sum_{j=1}^Nh_j((i-1)T)\\
  \ \ \ \ \ \ -\sum_{j=1}^Nh_j(iT),\ \ i=1, 2,\cdots,
\end{array}\right.
\end{equation}
where $g(y)=f(\cdot,y+y_e)-f(x,y_e)-f_y(x,y_e)y$.

{\bf Step 2. The well-posedness and the a-priori estimates.} It is known that the equation (\ref{e306}) admits local solution under Assumption $(H_f)$. We claim that:

\emph{For any given positive integer $K$, and positive real number $\eta$, there must be a number $\delta=\delta(K, \eta)>0$, such that when $\|z(0)\|_{1/2-\varepsilon}\leq \delta$, the solution to equation (\ref{e306}) exists on $[0, KT]$, and it satisfies that
\begin{equation}\label{e307}
\|z(t)\|_{1/2-\varepsilon}\leq \eta.
\end{equation}}
To prove the above claim, we firstly take $\delta< \min\{1, \eta\}$, and let $T_0=\inf\{t; \|z(t)\|_{1/2-\varepsilon}\geq \eta\}$. Our aim now is to choose $\delta$ small enough to avoid that $T_0<KT$ happens.

Suppose $T_0\in (0, T)$. Arbitrarily given $t\in (0, T_0]$. Multiplying the equation (\ref{e306}) by $\Delta^{1/2-\varepsilon}z$,  integrating on $\Omega\times (0, t)$, and applying the Cauchy-Schwarz inequality, we can obtain that
\begin{eqnarray}\label{e308}
&&\|z(t)\|_{1/2-\varepsilon}^2+\int_0^t \|\Delta^{3/4-\varepsilon/2}z(s)\|^2ds\nonumber\\
&\leq& C_1\|z(0)\|_{1/2-\varepsilon}+C_2\int_0^t\|z(s)\|^2_{1/2-\varepsilon}ds\nonumber\\
&&+\int_0^t\|\Delta^{-1/4-\varepsilon/2} g(z+\sum_{j=1}^Nh_j)\|^2 ds\nonumber\\
&\leq&  C_1\|z(0)\|_{1/2-\varepsilon}+C_2\int_0^t\|z(s)\|^2_{1/2-\varepsilon}ds\nonumber\\
&& + C_3\int_0^t \sum_{i=1}^l[\|z(s)\|^{2p_i}_{1/2-\varepsilon}+\sum_{j=1}^N\|h_j\|_{1/2-\varepsilon}^{2p_i}]ds\nonumber\\
&\leq& C_4(\|z(0)\|^2_{1/2-\varepsilon}+ \sum_{i=1}^l\|z(0)\|_{1/2-\varepsilon}^{2p_i})\nonumber\\
&&+C_5\sum_{i=1}^l\int_0^t \|z(s)\|^{2p_i}_{1/2-\varepsilon}ds\nonumber\\
&\leq& L_0\|z(0)\|^2_{1/2-\varepsilon}+\varphi(\eta)\int_0^t \|z(s)\|^2_{1/2-\varepsilon}ds,
\end{eqnarray}
where $L_0$ is a constant and $\varphi(\eta)=C_5\sum_{i=1}^l\eta^{2(p_i-1)}$. In the proof of (\ref{e308}), we used the fact that $\|h_j\|_{1/2-\varepsilon}\leq C \|z(0)\|_{1/2-\varepsilon}$, which follows from the definition of $h_j, j=1, \cdots, N$, and we also used the fact that, $\forall p_i, i=1, \cdots, l$,
\begin{eqnarray*}
&&\|\Delta^{-1/4-\varepsilon/2}z^{p_i}\|
\leq \sup\{|\langle z^{p_i}, \varphi \rangle|; \|\Delta^{1/4+\varepsilon/2}\varphi\|\leq 1\}\nonumber\\
&&\leq C \sup\{|\langle z^{p_i}, \varphi \rangle|; \|\varphi\|_{L^{p^*}(\Omega)}\leq 1\}\leq  C\|z^{p_i}\|_{L^{q^*}(\Omega)}\nonumber\\
&& \leq C\|z\|_{1/2-\varepsilon}^{p_i}, \forall z\in H^{1/2-\varepsilon}(\Omega).
\end{eqnarray*}
Here $p^*=\infty, q^*=1$ when $n=1$, and $p^*=\frac{2n}{n-1-2\varepsilon}, q^*=\frac{2n}{n+1+2\varepsilon}$ when $n>1$. The above inequalities follows by Sobolev imbedding Theorem and hypothesis $(H_f)$. We can infer by (\ref{e308}) and Gronwall's inequality that
\begin{equation}\label{e309}
\|z(t)\|_{1/2-\varepsilon}\leq L_0e^{\varphi(\eta)T/2}\|z(0)\|_{1/2-\varepsilon}, \forall t\in (0, T_0].
\end{equation}
Take $\delta_1:=\frac{\eta}{2L_0}e^{-\varphi(\eta)T/2}$, we see from (\ref{e309}) that, when $\delta\leq \delta_1$, $\|z(T_0)\|_{1/2-\varepsilon}\leq L_0e^{\varphi(\eta)T/2}\delta_1<\eta$. This leads to contradiction because $\|z(T_0)\|_{1/2-\varepsilon}\geq \eta$ by the definition of $T_0$. Hence, $T_0\geq T$ when $\delta\leq \delta_1$.

Suppose now that $T_0=T$. By the third equation in (\ref{e306}) and the equivalent form (\ref{e2113}), we see that
$$\|z(T)\|_{1/2-\varepsilon}\leq C_1\|z(T-)\|_{1/2-\varepsilon}+C_2\|z(0)\|_{1/2-\varepsilon}.$$
This together with inequality (\ref{e309}) imply that, there exists constant $L_1>\max\{1,L_0\}$ such that
\begin{equation}\label{e310}
\|z(T)\|_{1/2-\varepsilon}\leq L_1e^{\varphi(\eta)T/2}\|z(0)\|_{1/2-\varepsilon}.
\end{equation}
Take $\delta_2:=\frac{\eta}{2L_1}e^{-\varphi(\eta)T/2}$,  when $\delta\leq \delta_2$, we see from the above inequality that
 $\|z(T_0)\|_{1/2-\varepsilon}\leq L_1e^{\varphi(\eta)T/2}\delta_2<\eta$. This also leads to contradiction. Hence $T_0>T$ when $\delta\leq \delta_2$.

Following the same arguments as above, we can prove that, for any given $i\in \mathbb{N}$, when $\|z(iT)\|_{1/2-\varepsilon}\leq\delta_2$, it must holds that $\|z(t)\|_{1/2-\varepsilon}<\min\{\eta, L_1e^{\varphi(\eta)T/2}\|z(iT)\|_{1/2-\varepsilon}\}, \forall t\in [iT, (i+1)T]$. Therefore, we can see that if we take
\begin{equation}\label{e311}
\delta=\delta_{K, \eta}:=\frac{\eta}{2L_1^K}e^{-\varphi(\eta)KT/2},
\end{equation}
then $\|z(t)\|_{1/2-\varepsilon}<\eta, \forall t\in[0,KT]$, or equivalently, $T_0\geq KT$. The claim is proven. Moreover, we can infer from the above proof that, when $\|z(0)\|_{1/2-\varepsilon}\leq\delta_{K, \eta}$, the solution to equation (\ref{e306}) satisfies that
\begin{equation}\label{e312}
\|z(t)\|_{1/2-\varepsilon}\leq  L_1^Ke^{\varphi(\eta)KT/2}\|z(0)\|_{1/2-\varepsilon}, \forall t\in [0, KT].
\end{equation}

{\bf Step 3. The stability of the nonlinear closed-loop system} Write $\mathbf{z}^N(t)=Q_Nz(t)$, $z^s(t)=(I-P_N)z(t)$. Then, using the same argument for getting (\ref{e220}) and (\ref{e222}), we can obtain that, for each $i=0, 1, 2, \cdots$,
\begin{eqnarray*}
\mathbf{z}^N((i+1)T)&=&\sum_{k=1}^Ne^{-\gamma_kT}B_kB\mathbf{z}^N(iT)\\
&+&\int_{iT}^{(i+1)T}e^{-A_N((i+1)T-t)}Q_Ng(z+\sum_{j=1}^Nh_j)dt.
\end{eqnarray*}
It follows that
\begin{eqnarray*}
|B^{1/2}\mathbf{z}^N((i+1)T)|&\leq& e^{-\gamma_1T}|B^{1/2}\mathbf{z}^N(iT)|_N\\
&+&C\int_{iT}^{(i+1)T}|Q_Ng(z+\sum_{j=1}^Nh_j)|_Ndt.
\end{eqnarray*}
Then, for any given $K\in \mathbb{N}^+$, we see that, $\forall i=0, 1, 2, \cdots, K-1$,
\begin{eqnarray*}
|B^{1/2}\mathbf{z}^N((i+1)T)|&\leq& e^{-\gamma_1(i+1)T}|B^{1/2}\mathbf{z}^N(iT)|_N\\
&+&C(K)\int_{0}^{(i+1)T}|Q_Ng(z+\sum_{j=1}^Nh_j)|_Ndt.
\end{eqnarray*}
where $C(K)$ is a constant depending on $K$. This implies that
\begin{eqnarray}\label{e313}
&&\|P_Nz((i+1)T)\|_{1/2-\varepsilon}\leq  C_Ne^{-\gamma_1(i+1)T}\|P_Nz(0)\|_{1/2-\varepsilon}\nonumber\\
&&+C_1(K)\int_{0}^{(i+1)T}\|g(z+\sum_{j=1}^Nh_j)\|_{1/2-\varepsilon}dt.
\end{eqnarray}
On the other hand, the high frequency part $z_s$ can be written as
\begin{eqnarray*}
&&\mathbf{z}^s(KT)=e^{-A_s KT}\mathbf{z}^s(0)\nonumber\\&&+\sum_{i=1}^Ke^{-A_s (KT-iT)}[\sum_{j=1}^N((I-P_N))(h_j((i-1)T)-h_j(iT))]\nonumber\\
&&+ \int_{0}^{KT}e^{-A_s (KT-t)}(I-P_N)Rdt\nonumber\\&&+\int_{0}^{KT}e^{-A_s (KT-t)}(I-P_N)g(z+\sum_{j=1}^Nh_j)dt.
\end{eqnarray*}
Notice that $\|h_j(iT)\|_{1/2-\varepsilon}\leq C\|P_Nz(iT)\|_{1/2-\varepsilon}, \forall j=1, \cdots, N, i=0, 1, \cdots, N$,  we can infer by the above equality, the inequality (\ref{e313})and the same arguments for getting (\ref{e230})that
\begin{eqnarray}\label{e314}
\|\mathbf{z}^s(KT)\|_{1/2-\varepsilon}&\leq&C_se^{-\rho KT}\|z(0)\|_{1/2-\varepsilon}\nonumber\\&+&C_2(K)\int_{0}^{KT}\|g(z+\sum_{j=1}^Nh_j)\|_{1/2-\varepsilon}dt.
\end{eqnarray}
Letting $i=K-1$ in (\ref{e313}), it follows from (\ref{e312}), (\ref{e313}) and (\ref{e314}) that, when $\|z(0)\|_{1/2-\varepsilon}\leq \delta_{K, \eta}$ ($\delta_{K, \eta}$ is defined in (\ref{e311})),
\begin{eqnarray}\label{e315}
\|z(KT)\|_{1/2-\varepsilon}&\leq& L_2e^{-\rho KT}\|z(0)\|_{1/2-\varepsilon}\nonumber\\&+&C_3(K)e^{\varphi(\eta)KT/2}\varphi(\eta)\|z(0)\|_{1/2-\varepsilon}.
\end{eqnarray}
Now, we can take $K$ large enough such that $L_2e^{-\rho KT}\leq \frac{1}{2}e^{-\mu KT}$, and take $\eta$ small enough such that $C_3(K)e^{\varphi(\eta)KT/2}\varphi(\eta)\leq \frac{1}{2}e^{-\mu KT}$. Then, for such $K$ and $\eta$, when $\|z(0)\|_{1/2-\varepsilon}\leq \delta_{K, \eta}$, the following inequality holds
\begin{eqnarray}\label{e316}
\|z(KT)\|_{1/2-\varepsilon}&\leq& e^{-\mu KT}\|z(0)\|_{1/2-\varepsilon}.
\end{eqnarray}
By the same arguments for deriving (\ref{e312}) and (\ref{e316}), we can infer that, $\|z(t)\|_{Z}\leq C\|z(iKT)\|_{Z}, \forall t\in (iKT, (i+1)KT]$, and $\|z((i+1)KT)\|_Z\leq e^{-\mu KT}\|z(iKT)\|_{Z}$, $\forall i\in \mathds{N}^+$.  Now, arbitrarily given $t\in(0, +\infty)$, there must be $i\in \mathbb{N}^+$, such that $t\in (iKT, (i+1)KT]$, and we can infer that
\begin{eqnarray}\label{e317}
&&\|z(t)\|_{1/2-\varepsilon}\leq Ce^{-\mu KT}\|z(iKT)\|_{1/2-\varepsilon}\nonumber\\
&&\ \ \ \ \ \  \leq Ce^{-\mu iKT}\|z(0)\|_{1/2-\varepsilon}\leq \tilde{C}e^{-\mu t}\|z(0)\|_{1/2-\varepsilon}.
\end{eqnarray}

\end{proof}








\end{document}